\documentclass[11pt]{amsart}
\usepackage{graphicx}
\usepackage{mathabx} 
\usepackage{oldgerm}
\usepackage{mathdots}  
\usepackage{stmaryrd}
\usepackage{bm}
\usepackage[all]{xy}
\usepackage{color}
\usepackage{enumerate}

\usepackage{hyperref}
\usepackage{mathrsfs}
\usepackage{tikz} 
\usetikzlibrary{arrows,%
  decorations.markings,%
  matrix,%
  shapes}
  
\newcommand{\diagram}[3]{\matrix (#1) [matrix of math nodes,row
  sep={#2},column sep={#3},text height=1.5ex,text
  depth=0.25ex]}

\usepackage[latin1]{inputenc}
\usepackage{amsmath,amsthm, amscd, amssymb, amsfonts}

\numberwithin{equation}{section}

\newcommand{\End}{\mbox{\rm End\,}}
\theoremstyle{plain}

\numberwithin{equation}{section}

\newtheorem{theorem}[equation]{Theorem}
\newtheorem{corollary}[equation]{Corollary}

\newtheorem{lemma}[equation]{Lemma}

\theoremstyle{definition}
\newtheorem*{definition}{Definition}

\theoremstyle{definition}
\newtheorem{remark}[equation]{Remark}

\newcommand{\A}{\mathscr{A}}
\newcommand{\B}{\mathscr{B}}
\newcommand{\C}{\mathscr{C}}

\newcommand{\D}{\mathscr{D}}

\newcommand\Ext{\operatorname{Ext}}

\newcommand\cx{\operatorname{cx}}

\newcommand\e{\operatorname{e}}
\newcommand\op{\operatorname{op}}
\newcommand\ot{\otimes}
\renewcommand\mod{\operatorname{mod}}

\newcommand\Hom{\operatorname{Hom}}

\newcommand\Hoch{\operatorname{HH}}
\newcommand\HH{\Hoch}

\newcommand\unit{\mathbf{1}}
\DeclareMathOperator{\Ima}{Im}

\newcommand{\Coh}{\operatorname{H}\nolimits}

\makeatletter
\def\blx@maxline{77}
\makeatother

\begin{document}
\title[Cohomology of finite tensor categories]{On the cohomology of finite tensor categories}

\author[P.A.\ Bergh]{Petter Andreas Bergh}

\address{Petter Andreas Bergh \\ Institutt for matematiske fag \\ NTNU \\ N-7491 Trondheim \\ Norway} 
\email{petter.bergh@ntnu.no}


\begin{abstract}
It has been conjectured that finite tensor categories have finitely generated cohomology. We show that this is equivalent to finitely generated Hochschild cohomology for the endomorphism algebras of the projective generators.
\end{abstract}

\maketitle

\section{Introduction}

Given a finite group $G$ and a field $k$, the cohomology ring $\Coh^*(G,k)$ is a graded-commutative $k$-algebra. Moreover, for every left $kG$-module $M$, the tensor product induces a graded ring homomorphism from $\Coh^*(G,k)$ to $\Ext_{kG}^*(M,M)$, turning the latter into a module over the cohomology ring. By a classical result of Evens, Golod, and Venkov, the cohomology ring is a finitely generated $k$-algebra, and $\Ext_{kG}^*(M,M)$ is a finitely generated $\Coh^*(G,k)$-module for every finitely generated left $kG$-module $M$; cf.\ \cite{E,G,V}. One consequence of this is a powerful theory of support varieties for modules over such group algebras, where the varieties encode homological information. These support varieties are again a key ingredient in Benson, Carlson, and Rickard's classification in \cite{BCR} of thick tensor ideals in the stable module category.

It is an open question whether finitely generated cohomology holds for all finite-dimensional Hopf algebras, or even more generally, for all finite tensor categories. In \cite{EO}, Etingof and Ostrik conjectured that this is indeed true: for every finite tensor category, the cohomology ring of the unit object is finitely generated, and every object has finitely generated cohomology with respect to it. As for group algebras, finiteness ensures a fruitful theory of support varieties, recently explored in \cite{BEPW, BPW1, BPW2, BPW3}. 

The concept of finitely generated cohomology also applies to finite-dimensional algebras. For these, the Hochschild cohomology ring, which is always graded-commutative, plays the role of the cohomology ring. Again, given a left module, the tensor product induces a graded ring homomorphism into the cohomology of the module, turning this into a module over the Hochschild cohomology ring. However, not all finite-dimensional algebras have finitely generated cohomology. In fact, as shown in \cite{X}, there exist algebras whose Hochschild cohomology rings are not finitely generated. Finding classes of algebras for which finiteness holds is an ongoing program.

In this paper, we link the conjecture of Etingof and Ostrik to the Hochschild cohomology version for algebras. More precisely, we show that a finite tensor category has finitely generated cohomology if and only if finiteness with respect to Hochschild cohomology holds for the endomorphism ring of some (equivalently, every) projective generator in the category.

\subsection*{Acknowledgments} 
I would like to thank Dave Benson, Karin Erdmann, and Pavel Etingof for valuable comments and suggestions.

\section{Finite tensor categories and cohomology}\label{sec:prelim}

We start by recalling the basics on finite tensor categories and their cohomology. The standard reference for what follows is \cite{EGNO}. Let us fix a field $k$, not necessarily algebraically closed. 

By a finite tensor category over $k$ we mean a triple $\left ( \C, \ot, \unit \right )$ with the following properties. First, the category $\C$ is abelian, $k$-linear, and locally finite. Up to isomorphism, there are only finitely many simple objects in $\C$, and each of these has a projective cover. Next, the triple $\left ( \C, \ot, \unit \right )$ is a monoidal category, the unit object $\unit$ is simple, and the tensor product $\ot$ is bilinear on morphisms. Finally, the monoidal category is rigid: every object $X \in \C$ has a left dual $X^*$ and a right dual $^*X$. From now on, we fix such a finite tensor $k$-category $\left ( \C, \ot, \unit \right )$. 

Note that every object actually has a projective cover, not just the simple ones; see \cite[Remark 2.1(1)]{BEPW}. Therefore, with enough projective objects in $\C$, we may for all objects $X,Y \in \C$ define $\Ext_{\C}^n(X,Y)$ in two equivalent ways: as equivalence classes of $n$-fold extensions, or by using a projective resolution of $X$. We denote the graded $k$-vector space $\oplus_{n=0}^{\infty} \Ext_{\C}^n(X,Y)$ by $\Ext_{\C}^*(X,Y)$. The \emph{cohomology ring} of $\C$ is $\Ext_{\C}^*( \unit, \unit )$, which we denote by $\Coh^* ( \C )$. By \cite[Theorem 1.7]{SA}, this is a graded-commutative ring, that is, $xy = (-1)^{|x| \cdot |y|}yx$ for homogeneous elements $x,y$. With the tensor product being exact, the functor  $- \ot X$ induces a homomorphism
\begin{center}
\begin{tikzpicture}
\diagram{d}{3em}{3em}{
\Coh^* ( \C ) & \Ext_{\C}^*(X,X) \\
 };
\path[->, font = \scriptsize, auto]
(d-1-1) edge node{$\varphi_X$} (d-1-2);
\end{tikzpicture}
\end{center}
of graded $k$-algebras. Using $\varphi_X$ and $\varphi_Y$, then, and Yoneda composition, the $k$-vector space $\Ext_{\C}^*(X,Y)$ becomes both a right and a left $\Coh^* ( \C )$-module. However, by \cite[Corollary 2.3]{BPW2}, the two module structures coincide up to a sign for homogeneous elements. In particular, the image of the ring homomorphism $\varphi_X$ is contained in the graded center of $\Ext_{\C}^*(X,X)$.

In \cite{EO}, Etingof and Ostrik conjectured that $\Coh^* ( \C )$ is always finitely generated as a $k$-algebra, and that $\Ext_{\C}^*(X,X)$ is a finitely generated $\Coh^* ( \C )$-module for every object $X \in \C$. Our aim is to link this property of finite tensor categories to a similar property for finite-dimensional algebras, in terms of Hochschild cohomology. Recall therefore that if $A$ is a finite-dimensional $k$-algebra, then its Hochschild cohomology ring, denoted $\HH^* (A)$, is the graded $k$-algebra $\Ext_{A^{\e}}^*( A,A)$, where $A^e = A \ot_k A^{\op}$. As for $\Coh^* ( \C )$, this is a graded-commutative ring, a result proved already by Yoneda; see \cite[Proposition 3]{Y}. Given a finitely generated left $A$-module $M$, the tensor product $- \ot_A M$ induces a homomorphism
\begin{center}
\begin{tikzpicture}
\diagram{d}{3em}{3em}{
\HH^* (A) & \Ext_{A}^*(M,M) \\
 };
\path[->, font = \scriptsize, auto]
(d-1-1) edge node{$\psi_M$} (d-1-2);
\end{tikzpicture}
\end{center}
of graded $k$-algebras. Consequently, given two finitely generated left $A$-modules $M,N$, the Hochschild cohomology ring acts on $\Ext_A^*(M,N)$ in much the same way as $\Coh^* ( \C )$ acts on the cohomology of $\C$. Namely, $\Ext_A^*(M,N)$ becomes a right and a left $\HH^* (A)$-module via $\psi_M$ and $\psi_N$, respectively, followed by Yoneda composition. However, also here the two module structures coincide up to a sign for homogeneous elements; see \cite[Corollary 1.3]{SS}. Again, this means in particular that the image of the ring homomorphism $\psi_M$ is contained in the graded center of $\Ext_{A}^*(M,M)$.

\begin{definition}
Let $k$ be a field.

(1) A finite tensor category $\left ( \C, \ot, \unit \right )$ over $k$ satisfies the \emph{finiteness condition} \textbf{Fgt} if $\Coh^*( \C )$ is finitely generated, and $\Ext_{\C}^*(X,X)$ is a finitely generated $\Coh^*( \C )$-module for every object $X \in \C$.

(2) A finite-dimensional $k$-algebra $A$ satisfies the \emph{finiteness condition} \textbf{Fg} if $\HH^* (A)$ is finitely generated, and $\Ext_A^*(M,M)$ is a finitely generated $\HH^* (A)$-module for every finitely generated left $A$-module $M$.
\end{definition}

In our main result, we shall prove that \textbf{Fgt} holds for $\left ( \C, \ot, \unit \right )$ if and only if \textbf{Fg} holds for certain algebras associated to $\C$. However, before we start proving this, a couple of remarks are in order.

\begin{remark}\label{rem:fg}
(1) In the recent papers \cite{Be2, BEPW, BPW1, BPW2, BPW3}, the finiteness condition \textbf{Fgt} was denoted by \textbf{Fg}. However, when we study both finiteness conditions defined above, then we need to distinguish the two, and so we use \textbf{Fg} for the version involving algebras and Hochschild cohomology. This is historically correct, since this version was first introduced in \cite{EHSST}. We use the letter t in \textbf{Fgt} to indicate that we are considering finite tensor categories.

(2) The conjecture of Etingof and Ostrik states that \textbf{Fgt} holds for every finite tensor category, and as mentioned above, we shall prove that this is equivalent to the statement that \textbf{Fg} holds for certain finite-dimensional algebras. However, it is \emph{not} the case that \textbf{Fg} holds for all finite-dimensional algebras. In fact, this condition is quite strong, as it implies, for example, that the algebra is Gorenstein; see \cite[Theorem 2.5]{EHSST}. It is not even true that the Hochschild cohomology ring is finitely generated in general; a counterexample to this was first provided in \cite{X}.
\end{remark}

In the proof of our main result, we shall make use of the \emph{dual} category $\C^{\vee}$ of $\C$. This is the category with the same objects as $\C$, but with the directions of all the morphisms reversed. In category theory, this is usually called the opposite category, and denoted by $\C^{\rm{op}}$. However, for a monoidal category, like our $\C$, the opposite category $\C^{\rm{op}}$ usually refers to the monoidal category one obtains by switching the objects in the tensor product. It is a finite tensor category whenever $\C$ is.

\begin{remark}\label{rem:dual}
If $\left ( \C, \ot, \unit \right )$ is a finite tensor category over $k$, then so is $\left ( \C^{\vee}, \ot, \unit \right )$. Indeed, taking left duals, say, induces an equivalence $\C^{\vee} \longrightarrow \C^{\rm{op}}$. Moreover, if \textbf{Fgt} holds for $\left ( \C, \ot, \unit \right )$, then it also holds for $\left ( \C^{\vee}, \ot, \unit \right )$. Namely, to every homogeneous element of $\Coh^*( \C )$---that is, an extension---there corresponds a unique homogeneous element of $\Coh^*( \C^{\vee} )$, of the same degree, obtained by reversing the arrows. This induces an anti-isomorphism 
\begin{center}
\begin{tikzpicture}
\diagram{d}{3em}{3em}{
\Coh^* ( \C ) &\Coh^*( \C^{\vee} ) \\
 };
\path[->, font = \scriptsize, auto]
(d-1-1) edge (d-1-2);
\end{tikzpicture}
\end{center}
and so $\Coh^*( \C^{\vee} )$ is finitely generated, since $\Coh^*( \C )$ is. A similar argument shows that $\Ext_{\C^{\vee}}^*(X,X)$ is a finitely generated $\Coh^*( \C^{\vee} )$-module for every object $X \in \C^{\vee}$, when finite generation holds in $\C$.
\end{remark}

We shall need the following result, a generalized Eckmann-Shapiro lemma for finite tensor categories. The result is \cite[Lemma 2.4]{NP}, but we provide an alternative proof. A version for Hopf algebras was given in \cite[Theorem 9.3.9]{W}. Recall first that for objects $X,Y \in \C$, the graded $k$-vector space $\Ext_{\C}^*(X,Y)$ is a right $\Coh^*( \C )$-module via the tensor product induced ring homomorphism 
\begin{center}
\begin{tikzpicture}
\diagram{d}{3em}{3em}{
\Coh^* ( \C ) & \Ext_{\C}^*(X,X) \\
 };
\path[->, font = \scriptsize, auto]
(d-1-1) edge node{$\varphi_X$} (d-1-2);
\end{tikzpicture}
\end{center}
followed by Yoneda composition. Specifically, for homogeneous elements $\eta \in \Coh^*( \C )$ and $\theta \in \Ext_{\C}^*(X,Y)$, the right $\Coh^*( \C )$-module scalar product is given by the Yoneda composition $\theta \circ  ( \eta \ot X )$. Therefore, when $X = \unit$, then the right $\Coh^*( \C )$-module scalar product on $\Ext_{\C}^*( \unit, Y )$ is the Yoneda product directly.

\begin{lemma}[Generalized Eckmann-Shapiro lemma]\label{lem:ES}
If $\left ( \C, \ot, \unit \right )$ is a finite tensor category over $k$, then for all objects $X,Y \in \C$, the right $\Coh^*( \C )$-modules $\Ext_{\C}^*(X,Y)$ and $\Ext_{\C}^*( \unit, Y \ot X^* )$ are isomorphic.
\end{lemma}

\begin{proof}
By \cite[Proposition 2.10.8]{EGNO} the exact functors $- \ot X$ and $- \ot X^*$ form an adjoint pair $\left ( - \ot X, - \ot X^* \right )$. Now apply \cite[Theorem 2.1]{Be1} with the categories $\A$ and $\B$ replaced by our category $\C$, the functors $F$ and $G$ replaced by $- \ot X$ and $- \ot X^*$, respectively, and the object $X \in \A$ replaced by the unit object $\unit \in \C$. Note that \cite[Theorem 2.1]{Be1} is a generalized Eckmann-Shapiro lemma for abelian categories.
\end{proof}

We have now come to the main result. Recall first that, by definition, there are only finitely many simple objects of $\C$, up to isomorphism; let us denote them by $S_1, \dots, S_n$. Their projective covers $P_1, \dots, P_n$ are then indecomposable, and these are precisely the indecomposable projective objects of $\C$. The projective object $P_1 \oplus \cdots \oplus P_n$ is therefore a (minimal) projective generator of $\C$, in the sense that every projective object of $\C$ is a direct summand of a direct sum of copies of this object. Now let $P \in \C$ be any projective generator, and consider the finite-dimesional $k$-algebra $\End_{\C}(P)^{\op}$. As remarked in \cite[Remark 2.1]{BEPW}, the functor
\begin{center}
\begin{tikzpicture}
\diagram{d}{3em}{6em}{
\C & \mod \End_{\C}(P)^{\op} \\
 };
\path[->, font = \scriptsize, auto]
(d-1-1) edge node{$\Hom_{\C}(P,-)$} (d-1-2);
\end{tikzpicture}
\end{center}
is an (exact) equivalence of $k$-linear abelian categories, where $\mod \End_{\C}(P)^{\op}$ denotes the category of finitely generated left modules over $ \End_{\C}(P)^{\op}$. It follows from this that if $P$ and $Q$ are projective generators of $\C$, then $\End_{\C}(P)^{\op}$ and $\End_{\C}(Q)^{\op}$ are necessarily Morita equivalent.

Recall also that if $\left ( \C, \ot_{\C}, \unit_{\C} \right )$ and $\left ( \D, \ot_{\D}, \unit_{\D} \right )$ are two finite tensor categories, then their \emph{Deligne tensor product} $\C \boxtimes \D$ exists, and is again a finite tensor category. It is the $k$-linear abelian category universal with respect to right exact bifunctors on $\C \times \D$; see \cite{D} and \cite[Sections 1.11 and 4.6]{EGNO} for details. The canonical bifunctor from $\C \times \D$ to $\C \boxtimes \D$ 
is exact, and the image of an object $(X,Y) \in \C \times \D$ is denoted by $X \boxtimes Y$.

Finally, recall that if $V = \oplus_{n=0}^{\infty} V_n$ is a graded $k$-vector space of finite type (so that $\dim_k V_n$ is finite for all $n$), then its rate of growth is 
$$\gamma (V) \stackrel{\text{def}}{=} \inf \left \{ c \in \mathbb{N} \cup \{ 0 \} \mid \text{there exists } a \in \mathbb{R} \text{ with } \dim_k V_n \le an^{c-1} \text{ for } n \gg 0 \right \}$$
For a graded $k$-algebra $H = \oplus_{n=0}^{\infty} H_n$ of finite type, the Krull dimension of $H$ is now defined as $\gamma (H)$. When $H$ is graded-commutative and finitely generated as a $k$-algebra, then this is a finite integer. Namely, by the (graded-commutative version of the) Hilbert-Serre theorem, the Hilbert series $\sum_{n=0}^{\infty} ( \dim_k H_n ) x^n$ is a rational function of the form $f(x) / \prod_{i=1}^t ( x^{m_i}-1 )$, where $t$ is the number of homogeneous generators of $H$, and $m_1, \dots, m_t$ are their degrees. By \cite[Proposition 5.3.2]{B}, the order of the pole of the Hilbert series, at $x=1$, equals $\gamma (H)$, and so the latter is a finite integer. In this case, we can also define the Krull dimension in terms of chains of prime ideals for a certain commutative subalgebra. More precisely, define $H^{\bullet}$ to be just $H$ when the characteristic of $k$ is two, and set $H^{\bullet} = \oplus_{n=0}^{\infty} H_{2n}$ otherwise. This is then a graded $k$-algebra which is commutative in the ordinary sense, and it follows from \cite[Lemma 5.3.6]{B} that $\gamma ( H^{\bullet} ) = \gamma (H)$. By \cite[Lemma 5.4.6]{B}, the rate of growth $\gamma ( H^{\bullet} )$ equals the classical Krull dimension of $H^{\bullet}$ defined in terms of chains of prime ideals.

\begin{theorem}\label{thm:main}
Let $k$ be a field, $\left ( \C, \ot, \unit \right )$ a finite tensor category over $k$, and consider the following statements:
\begin{enumerate}
\item[\emph{(1)}] \emph{\textbf{Fgt}} holds for $\left ( \C, \ot, \unit \right )$;  
\item[\emph{(2)}] For some projective generator $P \in \C$, \emph{\textbf{Fg}} holds for the $k$-algebra $\End_{\C}(P)^{\op}$;
\item[\emph{(3)}] For every projective generator $P \in \C$, \emph{\textbf{Fg}} holds for the $k$-algebra $\End_{\C}(P)^{\op}$;
\item[\emph{(4)}] \emph{\textbf{Fg}} holds for some finite-dimensional $k$-algebra $A$ with $\mod A$ equivalent to $\C$;
\item[\emph{(5)}] \emph{\textbf{Fg}} holds for every finite-dimensional $k$-algebra $A$ with $\mod A$ equivalent to $\C$.
\end{enumerate}
Then \emph{(2)} $\Longrightarrow$ \emph{(1)}, and the four statements \emph{(2)}, \emph{(3)}, \emph{(4)}, \emph{(5)} are equivalent. Moreover, when the ground field $k$ is algebraically closed, then all five statements are equivalent, and the Krull dimension of $\Coh^* ( \C )$ equals that of $\HH^* \left ( \End_{\C}(P)^{\op} \right )$ and $\HH^* (A)$ for $P$ as in \emph{(2)}/\emph{(3)} and $A$ as in \emph{(4)}/\emph{(5)}.
\end{theorem}

\begin{proof}
As remarked before the theorem, given a projective generator $P$ of $\C$, the categories $\mod \End_{\C}(P)^{\op}$ and $\C$ are equivalent. In general, if $A$ and $B$ are two finite-dimensional $k$-algebras with both $\mod A$ and $\mod B$ equivalent to $\C$, then they are Morita equivalent. They are then also derived equivalent, and so it follows from \cite[Theorem]{KPS} that \textbf{Fg} holds for one of them if and only if it holds for the other. Therefore the statements (2), (3), (4), (5) are equivalent.

\sloppy Suppose now that (2) holds. Denote the $k$-algebra $\End_{\C}(P)^{\op}$ by just $A$, and take any module $M \in \mod A$. Since the Hochschild cohomology ring $\HH^* (A)$ is finitely generated, it is Noetherian as a ring, and then so is also the image of the graded ring homomorphism 
\begin{center}
\begin{tikzpicture}
\diagram{d}{3em}{3em}{
\HH^* (A) & \Ext_A^*(M,M) \\
 };
\path[->, font = \scriptsize, auto]
(d-1-1) edge node{$\psi_M$} (d-1-2);
\end{tikzpicture}
\end{center}
induced by the tensor product $- \ot_A M$. By assumption, the graded $k$-algebra $\Ext_A^*(M,M)$ is finitely generated as a left and as a right $\HH^* (A)$-module via $\psi_M$, hence it is also left and right module finite over the Noetherian subalgebra $\Ima \psi_M$. Consequently, it is a Noetherian ring itself. 

For every module $N \in \mod A$, the graded $k$-algebra $\Ext_A^*(M \oplus N,M \oplus N)$ is finitely generated as a left and as a right $\HH^* (A)$-module via $\psi_{M \oplus N}$. Since $\Ext_A^*(M,N)$ is a direct summand of $\Ext_A^*(M \oplus N,M \oplus N)$ as a left and a right module over $\HH^* (A)$, it, too, is finitely generated as such. But $\Ext_A^*(M,N)$ is a right $\HH^* (A)$-module via the ring homomorphism $\psi_M$ followed by Yoneda composition, and therefore it is finitely generated as a right module over $\Ext_A^*(M,M)$.

To sum up: for every module $M \in \mod A$, the $k$-algebra $\Ext_A^*(M,M)$ is a Noetherian ring, and $\Ext_A^*(M,N)$ is a finitely generated right $\Ext_A^*(M,M)$-module for every $N \in \mod A$. Now since the exact functor $\Hom_{\C}(P,-)$ induces an equivalence between the $k$-linear abelian categories $\C$ and $\mod A$, the same must hold in $\C$. That is, for every object $X \in \C$, the $k$-algebra $\Ext_{\C}^*(X,X)$ is a Noetherian ring, and $\Ext_{\C}^*(X,Y)$ is a finitely generated right $\Ext_{\C}^*(X,X)$-module for every $Y \in \C$. In particular, by taking $X = \unit$, we see that the cohomology ring $\Coh^* ( \C )$ is Noetherian, that is, finitely generated as a $k$-algebra. Moreover, given an object $X \in \C$, the right $\Coh^* ( \C )$-module $\Ext_{\C}^* ( \unit, X \ot X^* )$ is finitely generated. By Lemma \ref{lem:ES}, this implies that $\Ext_{\C}^* ( X,X )$ is also finitely generated as a right $\Coh^* ( \C )$-module, via the tensor product induced ring homomorphism $\varphi_X$ followed by Yoneda composition. This shows that (1) holds.

Suppose now that the ground field $k$ is algebraically closed, and that (1) holds, so that \textbf{Fgt} holds for $\left ( \C, \ot, \unit \right )$. Then by Remark \ref{rem:dual}, the same finiteness condition holds for the dual finite tensor category $\left ( \C^{\vee}, \ot, \unit \right )$. With $k$ being perfect, we can now apply \cite[Lemma 5.3]{NP}: the Deligne tensor product $\C \boxtimes \C^{\vee}$ satisfies \textbf{Fgt} as a finite tensor category over $k$. In other words, the cohomology ring $\Coh^*( \C \boxtimes \C^{\vee} )$ is a finitely generated $k$-algebra, and $\Ext_{\C \boxtimes \C^{\vee}}^*(X,X)$ is a finitely generated $\Coh^*( \C \boxtimes \C^{\vee} )$-module for every object $X \in \C \boxtimes \C^{\vee}$. The $\Coh^*( \C \boxtimes \C^{\vee} )$-module structure is via the ring homomorphism 
\begin{center}
\begin{tikzpicture}
\diagram{d}{3em}{3em}{
\Coh^*( \C \boxtimes \C^{\vee} ) & \Ext_{\C \boxtimes \C^{\vee}}^*(X,X) \\
 };
\path[->, font = \scriptsize, auto]
(d-1-1) edge node{$\varphi_X$} (d-1-2);
\end{tikzpicture}
\end{center}
induced by the tensor product (in $\C \boxtimes \C^{\vee}$)  with $X$, and so an argument similar to the one above shows the following: $\Ext_{\C \boxtimes \C^{\vee}}^*(X,X)$ is a Notherian ring, and $\Ext_{\C \boxtimes \C^{\vee}}^*(X,Y)$ is a finitely generated right $\Ext_{\C \boxtimes \C^{\vee}}^*(X,X)$-module for every object $Y \in \C \boxtimes \C^{\vee}$. Now if $Q \in \C \boxtimes \C^{\vee}$ is any projective generator, then the exact functor $\Hom_{\C \boxtimes \C^{\vee}}(Q,-)$ induces an equivalence between the $k$-linear abelian categories $\C \boxtimes \C^{\vee}$ and $\mod B$, where $B = \End_{\C \boxtimes \C^{\vee}}(Q)^{\op}$. Therefore, the same finiteness condition holds in $\mod B$: for every module $U \in \mod B$, the ring $\Ext_B^*(U,U)$ is Noetherian, and $\Ext_B^*(U,V)$ is a finitely generated right $\Ext_B^*(U,U)$-module for every $V \in \mod B$.

Let $S_1, \dots, S_n$ be the simple objects of $\C$, and take their projective covers $P_1, \dots, P_n$, the indecomposable projective objects. The $S_i$ are also the simple objects of the dual category $\C^{\vee}$. Moreover, since finite tensor categories are quasi-Frobenius (cf.\ \cite[Proposition 6.1.3]{EGNO}), and injective and projective objects change roles when passing to dual categories, the objects $P_1, \dots, P_n$ are also the indecomposable projective objects of $\C^{\vee}$. Therefore, there exists a permutation $\sigma$ of $\{ 1, \dots, n \}$ with the property that for each $i$, the projective cover of $S_i$ in $\C^{\vee}$ is $P_{\sigma (i)}$. 

Since $k$ is algebraically closed, the objects $S_i \boxtimes S_j$ are the simple objects of $\C \boxtimes \C^{\vee}$. Furthermore, since the canonical bifunctor from $\C \times \C^{\vee}$ to $\C \boxtimes \C^{\vee}$ is exact, there is, for each pair $i,j \in \{ 1, \dots, n \}$, an epimorphism 
\begin{center}
\begin{tikzpicture}
\diagram{d}{3em}{3em}{
P_i \boxtimes  P_{\sigma (j)} & S_i \boxtimes S_j \\ 
};
\path[->, font = \scriptsize, auto]
(d-1-1) edge (d-1-2);
\end{tikzpicture}
\end{center}
in $\C \boxtimes \C^{\vee}$. The object $P_i \boxtimes  P_{\sigma (j)}$ is projective; as can be seen for example by considering its cohomology with the simple objects. More precisely, let $S \boxtimes S'$ be one of the simple objects of $\C \boxtimes \C^{\vee}$. By \cite[Proposition 5.13(v)]{D}, there is an isomorphism
$$\Hom_{\C \boxtimes \C^{\vee}} ( P_i \boxtimes  P_{\sigma (j)}, S \boxtimes S' )  \simeq \Hom_{\C} ( P_i, S ) \otimes_k \Hom_{\C^{\vee}} ( P_{\sigma (j)}, S' )$$
and one similarly obtains an isomorphism
$$\Ext_{\C \boxtimes \C^{\vee}}^* ( P_i \boxtimes  P_{\sigma (j)}, S \boxtimes S' )  \simeq \Ext_{\C}^* ( P_i, S ) \otimes_k \Ext_{\C^{\vee}}^* ( P_{\sigma (j)}, S' )$$
As $P_i$ and $P_{\sigma (j)}$ are projective in $\C$ and $\C^{\vee}$, we see that $\Ext_{\C \boxtimes \C^{\vee}}^i ( P_i \boxtimes  P_{\sigma (j)}, S \boxtimes S' ) = 0$ for $i \ge 1$. Thus the object $P_i \boxtimes  P_{\sigma (j)}$ is projective in $\C \boxtimes \C^{\vee}$. The epimorphism from this object to the simple object $S_i \boxtimes S_j$ shows it has the projective cover of the latter as a direct summand.

Consider the projective generator $P = P_1 \oplus \cdots \oplus P_n$ of $\C$. By the above, and the fact that the simple objects of $\C \boxtimes \C^{\vee}$ are the $n^2$ objects $S_i \boxtimes S_j$ for $i,j \in \{ 1, \dots, n \}$, we see that $P \boxtimes P \simeq \oplus_{i,j} P_i \boxtimes  P_{\sigma (j)}$ is a projective generator of $\C \boxtimes \C^{\vee}$. Therefore, from earlier in the proof, the $k$-algebra $B = \End_{\C \boxtimes \C^{\vee}} ( P \boxtimes  P )^{\op}$ satisfies the following: for every module $U \in \mod B$, the ring $\Ext_B^*(U,U)$ is Noetherian, and $\Ext_B^*(U,V)$ is a finitely generated right $\Ext_B^*(U,U)$-module for every $V \in \mod B$. Now, there are isomorphisms
\begin{eqnarray*}
\End_{\C \boxtimes \C^{\vee}} ( P \boxtimes  P )^{\op}  & \simeq & \left ( \End_{\C} ( P ) \otimes_k \End_{\C^{\vee}} ( P ) \right )^{\op} \\
& \simeq & \End_{\C} ( P )^{\op} \otimes_k \End_{\C^{\vee}} ( P )^{\op} \\
& \simeq & \End_{\C} ( P )^{\op} \otimes_k \End_{\C} ( P )
\end{eqnarray*}
of $k$-algebras, giving $B \simeq A \otimes_k A^{\op} = A^{\e}$ for $A = \End_{\C} ( P )^{\op}$. Consequently, the Hochschild cohomology ring $\HH^* (A) = \Ext_{A^{\e}}^*( A,A)$ is Noetherian, that is, finitely generated. Moreover, for every module $V \in \mod A^{\e}$, that is, a finitely generated $A$-bimodule, the right $\HH^* (A)$-module $\Ext_{A^{\e}}^*( A,V)$ is finitely generated. Therefore, by \cite[Lemma 3.2(1)]{Be1}, the finiteness condition \textbf{Fg} holds for the $k$-algebra $A$. 

Finally, suppose that $k$ is algebraically closed, and that the Krull dimension of $\Coh^* ( \C )$ is $d$, say. For an object $X \in \C$, we denote its complexity by $\cx_{\C} (X)$; this is equal to the rate of growth of the lengths of the object in its minimal projective resolution. It follows from \cite[Theorem 2.11(1)]{BPW2} that $\cx_{\C} (X) \le \cx_{\C} ( \unit ) =d$ for every $X \in \C$, so that 
$$\max \{ \cx_{\C} (X) \mid X \in \C \} =d$$
Now let $P$ be any projective generator of $\C$, and denote $\End_{\C}(P)^{\op}$ by $A$ as before. Since $\C$ and $\mod A$ are equivalent as abelian categories, we obtain
$$\max \{ \cx_{A} (M) \mid M \in \mod A \} =d$$
where the complexity $\cx_{A} (M)$ of an $A$-module $M$ is defined in the same way as for objects of $\C$. The maximal complexity is always obtained by at least one of the simple modules, so that $\cx_{A} ( A / \mathfrak{r} ) =d$, where $\mathfrak{r}$ is the Jacobson radical of $A$. It follows from the proof of \cite[Corollary 3.6]{Be0} that $\cx_{A} ( A / \mathfrak{r} )$ equals the Krull dimension of $\HH^* (A)$, which must therefore be $d$. Finally, it is well known that derived equivalent algebras have isomorphic Hochschild cohomology rings; see \cite[Theorem 4.2]{H} and \cite[Proposition 2.5]{R}. This concludes the proof.
\end{proof}

Let us look at some consequences. The first one concerns finite generation of cohomology over finite-dimensional Hopf algebras. For such an algebra $A$, the category $\mod A$ is a finite tensor category, so Theorem \ref{thm:main} applies. 

\begin{corollary}\label{cor:Hopf}
Let $A$ be a finite-dimensional Hopf algebra. If \emph{\textbf{Fg}} holds for $A$, then \emph{\textbf{Fgt}} holds for $\mod A$, and the converse holds when the ground field is algebraically closed.
\end{corollary}

It should be noted that, in this case, the result is actually true without the assumption that $k$ is algebraically closed; see \cite[Lemma 4.2]{Be1}, which is a strengthening of \cite[Proposition 3.4]{NWW}. As an example, let $k$ be a field, $G$ a finite group, and consider the group algebra $kG$. This is a finite-dimensional Hopf algebra, and \textbf{Fgt} holds for $\mod kG$ by classical results of Evens, Golod, and Venkov; cf.\ \cite{E,G,V}. Therefore \textbf{Fg} holds for $kG$; see, for example, \cite[Example 3.2]{W1}.

Finally, recall that an abelian category is of \emph{finite representation type} if it contains only finitely many isomorphism classes of indecomposable objects. The following result shows that \textbf{Fgt} holds for such finite tensor categories, over algebraically closed fields.

\begin{corollary}\label{cor:finreptype}
If $\left ( \C, \ot, \unit \right )$ is a finite tensor category of finite representation type over an algebraically closed field, then \emph{\textbf{Fgt}} holds.
\end{corollary}

\begin{proof}
Let $P$ be a projective generator of $\C$, and consider the $k$-algebra $A =  \End_{\C} ( P )^{\op}$. We know that $\C$ and $\mod A$ are equivalent as abelian categories, hence $A$ has finite representation type. Moreover, since $\C$ is a quasi-Frobenius category, so is $\mod A$, that is, the algebra $A$ is selfinjective. Therefore, as shown in \cite{Du}, the algebra is periodic as a bimodule: there exists a positive integer $n$ for which $\Omega_{A^{\e}}^n (A) \simeq A$ (this is where we need to assume that the ground field is algebraically closed). This isomorphism corresponds to a homogeneous element $\eta \in \HH^* (A)$ of degree $n$, with the property that $\HH^* (A)$ is finitely generated as a module over the polynomial subalgebra $k[ \eta ]$. Hence $\HH^* (A)$ is a finitely generated $k$-algebra (see also \cite[Corollary 2.14]{ES}). Moreover, it is not hard to see that $\Ext_A^*(M,M)$ is a finitely generated $\HH^* (A)$-module for every $M \in \mod A$. In fact, every indecomposable non-projective such $M$ is periodic of period dividing $n$, and $\Ext_A^*(M,M)$ is a finitely generated as a module over $k[ \eta ]$. Therefore \textbf{Fg} holds for the algebra $A$, and then by Theorem \ref{thm:main}, \textbf{Fgt} holds for $\left ( \C, \ot, \unit \right )$.
\end{proof}

%%%%%%%%%%%%%%%%%%%%%%%%%%%%%%%%%%%%%%%%%%%%%%%%%%

%%%%%%%%%%%%%%%%%%%%%%%%%%%%%%%%%%%%%%%%%%%%%%%%


\begin{thebibliography}{999}

\bibitem{B}
D.J.\ Benson, \emph{Representations and cohomology. II. Cohomology of groups and modules}, Cambridge Stud.\ Adv.\ Math., 31, Cambridge University Press, Cambridge, 1991, x+278 pp.

\bibitem{BCR}
D.J.\ Benson, J.\ Carlson, J.\ Rickard, \emph{Thick subcategories of the stable module category}, Fund.\ Math.\ 153 (1997), no.\ 1, 59--80.

%\bibitem{BE}
%D.J.\ Benson, P.\ Etingof, \emph{On cohomology in symmetric tensor categories in prime characteristic}, Homology, Homotopy and Applications 24 (2022), no.\ 2, 163--193.

%\bibitem{BEO}
%D.J.\ Benson, P.\ Etingof, V.\ Ostrik, \emph{New incompressible symmetric tensor categories in positive characteristic}, Duke Math.\ J.\ 172 (2023), no.\ 1, 105--200.

\bibitem{Be0}
P.A.\ Bergh, \emph{Representation dimension and finitely generated cohomology}, Adv.\ Math.\ 219 (2008), no.\ 1, 389--400.

\bibitem{Be1}
P.A.\ Bergh, \emph{Separable equivalences, finitely generated cohomology and finite tensor categories}, Math.\ Z.\ 304 (2023), no.\ 3, Paper No.\ 49, 21 pp.

\bibitem{Be2}
P.A.\ Bergh, \emph{Homology of complexes over finite tensor categories}, J.\ Noncommut.\ Geom.\ 19 (2025), no.\ 1, 249--268.

\bibitem{BEPW}
P.A.\ Bergh, K.\ Erdmann, J.Y.\ Plavnik, S.\ Witherspoon, \emph{On the representation type of a finite tensor category}, preprint.

\bibitem{BPW1}
P.A.\ Bergh, J.Y.\ Plavnik, S.\ Witherspoon, \emph{Support varieties for finite tensor categories: complexity, realization, and connectedness}, J.\ Pure Appl.\ Algebra 225 (2021), no.\ 9, Paper No.\ 106705, 21 pp.

\bibitem{BPW2}
P.A.\ Bergh, J.Y.\ Plavnik, S.\ Witherspoon, \emph{Support varieties for finite tensor categories: the tensor product property}, Ann.\ Repr.\ Th.\ 1 (2024), 4, p.\ 539--566. 

\bibitem{BPW3}
P.A.\ Bergh, J.Y.\ Plavnik, S.\ Witherspoon, \emph{Suport varieties without the tensor product property}, Bull.\ London Math.\ Soc.\ 56 (2024), 2150--2161.

%\bibitem{CEO}
%K.\ Coulembier, P.\ Etingof, V.\ Ostrik, \emph{Incompressible tensor categories}, Adv.\ Math.\ 457 (2024), Paper No.\ 109935, 65 pp.

\bibitem{D}
P.\ Deligne, \emph{Cat{\'e}gories tannakiennes}, (French) [Tannakian categories], in \emph{The Grothendieck Festschrift, Vol.\ II}, 111--195, Progr.\ Math., 87, Birkh{\"a}user Boston, Boston, MA, 1990.

\bibitem{Du}
A.S.\ Dugas, \emph{Periodic resolutions and self-injective algebras of finite type}, J.\ Pure Appl.\ Algebra 214 (2010), no.\ 6, 990--1000.

\bibitem{EHSST}
K.\ Erdmann, M.\ Holloway, N.\ Snashall, {\O}.\ Solberg, R.\ Taillefer, \emph{Support varieties for selfinjective algebras}, K-Theory 33 (2004), no.\ 1, 67--87.

\bibitem{ES}
K.\ Erdmann, A.\ Skowro{\'n}ski, \emph{Periodic algebras}, in \emph{Trends in representation theory of algebras and related topics}, 201--251, EMS Ser.\ Congr.\ Rep., European Mathematical Society (EMS), Z{\"u}rich, 2008.

\bibitem{EGNO} 
P.\ Etingof, S.\ Gelaki, D.\ Nikshych, V.\ Ostrik, \emph{Tensor categories}, Math.\ Surveys Monogr., 205, American Mathematical Society, Providence, RI, 2015, xvi+343 pp.
 
\bibitem{EO} 
P.\ Etingof, V.\ Ostrik, \emph{Finite tensor categories}, Mosc.\ Math.\ J.\ 4 (2004), no.\ 3, 627--654, 782--783. 

\bibitem{E}
L.\ Evens, \emph{The cohomology ring of a finite group}, Trans.\ Amer.\ Math.\ Soc.\ 101 (1961), 224--239.

\bibitem{G}
E.\ Golod, \emph{The cohomology ring of a finite $p$-group} (Russian), Dokl.\ Akad.\ Nauk SSSR 125 (1959), 703--706.

\bibitem{H}
D.\ Happel, \emph{Hochschild cohomology of finite-dimensional algebras}, Lecture Notes in Math.\ 1404, Springer-Verlag, Berlin, 1989, 108--126.

\bibitem{KPS}
J.\ K{\"u}lshammer, C.\ Psaroudakis, {\O}.\ Skarts{\ae}terhagen, \emph{Derived invariance of support varieties}, Proc.\ Amer.\ Math.\ Soc.\ 147 (2019), no.\ 1, 1--14.

\bibitem{NP}
C.\ Negron, J.Y.\ Plavnik, \emph{Cohomology of finite tensor categories: duality and Drinfeld centers}, Trans.\ Amer.\ Math.\ Soc. 375 (2022), 2069--2112.

\bibitem{NWW}
V.C.\ Nguyen, X.\ Wang, S.\ Witherspoon, \emph{New approaches to finite generation of cohomology rings}, J.\ Algebra 587 (2021), 390--428.

\bibitem{R}
J.\ Rickard, \emph{Derived equivalences as derived functors}, J.\ London Math.\ Soc.\ (2) 43 (1991), no.\ 1, 37--48.

\bibitem{SS}
N.\ Snashall, {\O}.\ Solberg, \emph{Support varieties and Hochschild cohomology rings}, Proc.\ London Math.\ Soc.\ (3) 88 (2004), no.\ 3, 705--732.

\bibitem{SA}
M.\ Su\'arez-\'Alvarez, \emph{The Hilton-Eckmann argument for the anti-commutativity of cup products}, Proc.\ Amer.\ Math.\ Soc.\ 132 (2004), no.\ 8, 2241--2246.

\bibitem{V}
B.B.\ Venkov, \emph{Cohomology algebras for some classifying spaces} (Russian), Dokl.\ Akad.\ Nauk SSSR 127 (1959), 943--944.

\bibitem{W1}
S.J.\ Witherspoon, \emph{Varieties for modules of finite dimensional Hopf algebras}, in \emph{Geometric and topological aspects of the representation theory of finite groups}, 481--495, Springer Proc.\ Math.\ Stat.\ 242, Springer, Cham, 2018.

\bibitem{W}
S.J.\ Witherspoon, \emph{Hochschild cohomology for algebras}, Grad.\ Stud.\ Math., 204, American Mathematical Society, Providence, RI, 2019, xi+250 pp.

\bibitem{X}
F.\ Xu, \emph{Hochschild and ordinary cohomology rings of small categories}, Adv.\ Math.\ 219 (2008), no.\ 6, 1872--1893.

\bibitem{Y}
N.\ Yoneda, \emph{Note on products in $\Ext$}, Proc.\ Amer.\ Math.\ Soc.\ 9 (1958), 873--875.

%\nocite{*}
%\printbibliography

\end{thebibliography}
\end{document}